\documentclass[a4paper, reqno]{amsart}

\title{Hessian of the Busemann function on Damek-Ricci spaces}
\date{\today}

\author[H.~Satoh]{Hiroyasu Satoh}
\address{Liberal Arts and Sciences, Nippon Institute of Technology, 4-1 Gakuendai, Miyashiro-machi, Saitama 345-8501, Japan}
\email{hiroyasu@nit.ac.jp}


\usepackage{amsmath,amssymb,amsthm,amscd,bm}
\allowdisplaybreaks

\def\ad{\mathop{\mathrm{ad}}\nolimits}
\def\ker{\mathop{\mathrm{Ker}}\nolimits}
\def\id{\mathop{\mathrm{id}}\nolimits}

\theoremstyle{definition}
\newtheorem{thm}{Theorem}
\newtheorem{lem}[thm]{Lemma}

\newtheorem{prop}[thm]{Proposition}
\newtheorem{Def}[thm]{Definition}
\newtheorem{rem}[thm]{Remark}

\begin{document}
\maketitle

\begin{abstract}
In this note, we calculate the Hessian $H_\theta=\nabla d b_\theta$ of the Busemann function $b_\theta$ on a Damek-Ricci space.
We investigate eigenvalues of $H_\theta$ and show its positive definiteness.
\end{abstract}

\section{Introduction}

A Damek-Ricci space is a Riemannian manifold formed by a one-dimensional extension of a nilpotent Lie group known as the generalized Heisenberg group.
This class of manifolds includes the complex hyperbolic space, quaternionic hyperbolic space, and the octonionic hyperbolic plane.
As a special case, the real hyperbolic space can also be regarded as a Damek-Ricci space.  

A Riemannian manifold $M$ is called a harmonic manifold if the volume density function of any geodesic sphere in $M$ depends only on the radius.
A conjecture about harmonic manifolds states that \textit{harmonic manifolds are either rank one symmetric spaces or flat Euclidean spaces}.
This conjecture, known as the Lichnerowicz conjecture, has been proven true in the 4-dimensional case and the compact case.
However, the Damek-Ricci spaces provide counterexamples to this conjecture in the non-compact case.
For details on the Lichnerowicz conjecture, see \cite{K}.

The Busemann function $b_\gamma$, determined by a geodesic $\gamma$, acts like a distance function from infinity, plays a crucial role in the analysis of non-compact spaces.  
In spaces such as a Hadamard manifold $M$, where an ideal boundary $\partial M$ can be defined, the Busemann function can be regarded as a function on $M$ determined by a boundary point $\theta \in \partial M$, provided that a base point is fixed.
The Busemann function is significant for analyzing the asymptotic behavior of geodesics and for understanding the structure of the ideal boundary.
 
A level hypersurface of a Busemann function is called a horosphere.  
The shape operator of a horosphere is determined by the Hessian of the Busemann function.
Note that while the Hessian of a function is a $(0,2)$-tensor field,
it can be naturally regarded as a $(1,1)$-tensor field via the metric $g$, as is customary.
In rank one symmetric spaces of non-compact type,
the Hessian $H_\theta$ of the Busemann function can be expressed as
$$
H_\theta=g-db_\theta\otimes db_\theta+\sum_{i=1}^d db_\theta\circ J_i\otimes db_\theta\circ J_i,
$$
where $\{J_i\}_{i=1,\ldots,d}$ is the associated complex ($d=1$) or quaternionic ($d=3$), octonionic ($d=7$) structure.
In the case of real hyperbolic space, the summation term vanishes.
This indicates that horospheres are hypersurfaces with constant principal curvatures,
and their spectra are common for all horospheres.  
\textit{Does this property characterize rank one symmetric spaces of non-compact type as the only non-compact Riemannian manifolds with such features?}
In \cite{IS2013}, we investigated this question and obtained partial results.  
For example, suppose all horospheres in $M$ are umbilic at every point,
and the values of the principal curvatures are the same for all horospheres.
In that case, $M$ is isometric to real hyperbolic space \cite[Theorem 1.1]{IS2013}.

In this paper, we compute the Hessian of the Busemann function on Damek-Ricci spaces and analyze the properties of its spectrum.

\begin{thm}\label{main0}
Let $S$ be a Damek-Ricci space, which is a one-dimensional extension of the generalized Heisenberg group $N=\exp(\frak{v}\oplus\frak{z})$ and $\partial S\cong N\cup\{\infty\}$ be the ideal boundary of $S$.
Let $b_\theta$ be the Busemann function on $S$ determined by a boundary point $\theta\in\partial S$ and $H_\theta$ be its Hessian.
\begin{itemize}
\item[(1)] When $\theta = \infty$, $H_\theta$ has eigenvalues $0, \dfrac{1}{2}$ and $1$ at any point $p \in S$, with multiplicities $1, k=\dim \frak{v}$ and $m=\dim \frak{z}$,
respectively.
\item[(2)]  When $\theta = (v, z)\ (\ne \infty)$ and $p=(V, Z, a)\in S\cong \frak{v}\times\frak{z}\times\mathbb{R}_+$,
\begin{itemize}
\item[(2-i)] if $\mathcal{V}=v-V=0$ or $\mathcal{Z}=z-Z-\dfrac{1}{2}[V, v]=0$,
then $H_\theta$ has eigenvalues $0, \dfrac{1}{2}$ and $1$ at $p$, with multiplicities $1, k$ and $m$,
respectively.
\item[(2-ii)] if $\mathcal{V}\ne 0$ and $\mathcal{Z}\ne 0$, then
\begin{itemize}
\item[(2-ii-a)] The restriction $H_\theta|_{\frak{s}_4}$ has eigenvalues $0, \dfrac{1}{2}$ and $1$ at $p$, with multiplicities $1, 2$ and $1$,
respectively.
Here $\frak{s}_4$ is a subspace of $T_pS$ spaned by $\{\mathcal{V}, J_\mathcal{Z}\mathcal{V}, \mathcal{Z}, A\}$.
\item[(2-ii-b)] The eigenvalues $\lambda$ of $H_\theta|_{\frak{s}_4^\perp}$ are solutions of the following cubic equation
\begin{equation}\label{mu0}
\left(\lambda-\dfrac{1}{2}\right)^2(\lambda-1)
=-(\mu+1)\frac{a^2|\mathcal{V}|^4|\mathcal{Z}|^2}{8F^3}
\end{equation}
and these are all positive.
Here $\mu$ is an eigenvalue of the symmetric operator $K^2$ (see Section 2.2 for the definition), and $F$ is a positive function on $S\times N$ (see Section 3.1.2 for the definition).
\end{itemize}
\end{itemize}
\end{itemize}
\end{thm}
This theorem shows that if all horospheres in a Damek-Ricci space $S$ have constant principal curvatures, then $S$ is a rank one symmetric space of non-compact type.

\begin{rem}
We note that the foundational aspects of the Hessian computation in \cite{IKPS} are based on earlier results from a preliminary version of this work.
In \cite{IKPS}, the focus is on the spectral analysis of the Hessian in connection with the rank of geodesics and the visibility axiom, utilizing a slightly different framework and notation.
In contrast, this paper revisits the Hessian's spectral properties from a broader perspective, further clarifying the eigenvalue structure and its implications in the geometry of Damek-Ricci spaces.
These differences in approach provide complementary insights into the problem.
\end{rem}

The paper is organized as follows.
In Section 2, we provide the necessary preliminaries, including a review of generalized Heisenberg groups, the operator $K$, and the definition of Damek-Ricci spaces.
Section 3 focuses on the Hessian $H_\theta$ of the Busemann function on Damek-Ricci spaces, where we analyze its components and eigenvalues in detail for both $\theta=\infty$ and $\theta\ne\infty$.
The proof of the main theorem, Theorem 1, is presented here, including the spectral properties of $H_\theta$.

\section{Preliminaries}

\subsection{Generalized Heisenberg groups}

Let $(\frak{n}, [\cdot,\cdot]_\frak{n})$ denote a 2-step nilpotent Lie algebra
with a positive definite inner product $\langle\cdot,\cdot\rangle_\frak{n}$.
Let $\frak{z}$ denote the center of $\frak{n}$ and $\frak{v}$ its orthogonal complement.
For $Z\in \frak{z}$, we define a skew-symmetric linear map
$J_Z : \frak{v}\rightarrow \frak{v}$ by $\langle J_Z V_1, V_2\rangle_\frak{n}=\langle Z, [V_1,V_2]_\frak{n}\rangle_\frak{n}$ for $V_1, V_2\in \frak{v}$.
If for every $Z\in\frak{z}$,
\begin{equation}\label{def_gha}
(J_Z)^2=-|Z|^2\mathrm{id}_\frak{v}
\end{equation}
holds, then we say that $\frak{n}=(\frak{n},  [\cdot,\cdot]_\frak{n},  \langle\cdot,\cdot\rangle_\frak{n})$ is a {\it generalized Heisenberg algebra}.
Here $\mathrm{id}_\frak{v}$ is the identity map on $\frak{v}$.

\begin{lem}[{\cite[p.24--25]{BTV}}]\label{gHprop}
Let $(\frak{n}, [\cdot,\cdot]_\frak{n},  \langle\cdot,\cdot\rangle_\frak{n})$ be a generalized Heisenberg algebra.
For $V, V_i\in\frak{v}$ and $Z, Z_i\in\frak{z}$ ($i=1, 2$),
the following equations hold:
\begin{enumerate}
\item $J_{Z_1} \circ J_{Z_2}+J_{Z_2} \circ J_{Z_1}=-2\langle Z_1, Z_2\rangle \mathrm{id}_\frak{v}$
\item $\langle J_{Z_1} V_1, J_{Z_2} V_2\rangle+\langle J_{Z_2} V_1, J_{Z_1} V_2\rangle
=2\langle V_1, V_2\rangle\langle Z_1, Z_2\rangle$
\item $\langle J_Z V_1, J_Z V_2\rangle=|Z|^2\langle V_1, V_2\rangle$
\item $\langle J_{Z_1} V, J_{Z_2} V\rangle=|V|^2\langle Z_1, Z_2\rangle$
\item $[J_Z V_1, V_2]-[V_1, J_Z V_2]=-2\langle V_1, V_2\rangle Z$
\item $[J_{Z_1} V, J_{Z_2} V]=[V, J_{Z_1}J_{Z_2} V]$
\item $[J_Z V_1, J_Z V_2]=-|Z|^2[V_1, V_2]-2\langle V_1, J_Z V_2\rangle Z$
\item $[V, J_Z V]=|V|^2 Z$
\end{enumerate}
\end{lem}

\begin{proof}
Using the polarization identity, (i) can be derived from \eqref{def_gha}.
(ii) is obtained from the skew-symmetry of $J_Z$ and (i).
In (ii), replacing $Z_1$ and $Z_2$ with $Z$ yields (iii).  
Similarly, in (ii), replacing $V_1$ and $V_2$ with $V$ yields (iv).

Equations (v) and (i) are equivalent.  
In (v), replacing $Z$ with $Z_1$, $V_1$ with $V$, and $V_2$ with $J_{Z_2} V$ yields (vi).  
In (v), replacing $V_2$ with $J_Z V_2$ also yields (vi).  
Additionally, in (v), replacing both $V_1$ and $V_2$ with $V$ yields (vi).
\end{proof}

We denote the forml adjoint map of $\ad(V)=[V,\,\cdot\,] : \frak{v}\rightarrow\frak{z}$ by $\ad(V)^\ast$.
Then, (v) and (viii) in lemma \ref{gHprop} are written as
\begin{equation}
\ad(V_1)\circ \ad(V_2)^\ast
+\ad(V_2)\circ \ad(V_1)^\ast
=2\langle V_1, V_2\rangle \mathrm{id}_{\frak{z}},
\end{equation}
\begin{equation}
\ad(V)\circ \ad(V)^\ast
=|V|^2 \mathrm{id}_{\frak{z}}.
\end{equation}

Fix a non-zero vector $V\in \frak{v}$, we obtain an orthogonal decomposition of $\frak{v}$ as follows:
\begin{equation}\label{decompv}
\frak{v}=\ker(\ad(V))\oplus J_\frak{z} V,\quad
J_\frak{z} V:= \{J_Y V\,|\,Y\in \frak{z}\}.
\end{equation}
Since $V\in\ker(\ad(V))$, $\ker(\ad(V))$ can be decomposed as the orthogonal direct sum $\mathbb{R}V$ and its orthogonal complement in $\ker(\ad(V))$, denoted by $\ker(\ad(V))_0$:
\begin{equation}
\frak{v}=\mathbb{R}V\oplus\ker(\ad(V))_0\oplus J_\frak{z} V.
\end{equation}
Moreover, fix  a non-zero vector $Z\in \frak{z}$, we obtain an orthogonal decomposition
\begin{equation}
J_\frak{z} V=\mathbb{R}J_Z V \oplus J_{Z^\perp} V,\quad
Z^\perp=\{Y\in \frak{z}\,|\,\langle Z, Y\rangle_{\frak{n}}=0\}.
\end{equation}
Hence we have
\begin{equation}\label{decomp1}
\mathfrak{v}=\mathbb{R}V\oplus\mathbb{R}J_Z V\oplus\ker(\ad(V))_0\oplus J_{Z^\perp} V.
\end{equation}
By using $J_Z V$ instead of $V$, we obtain
\begin{equation}\label{decomp2}
\mathfrak{v}=\mathbb{R}V\oplus\mathbb{R}J_Z V\oplus\ker(\ad(J_Z V))_0\oplus J_{Z^\perp} (J_Z V).
\end{equation}
From \eqref{decomp1} and \eqref{decomp2}, we have
\begin{equation}\label{decomp3}
\mathfrak{v}=\mathbb{R}V\oplus\mathbb{R}J_Z V\oplus
\mathrm{span}\{J_{Z^\perp} V, J_{Z^\perp} (J_Z V)\}\oplus
\left(
\ker(\ad(V))_0\cap \ker(\ad(J_Z V))_0
\right).
\end{equation}

If $\frak{n}$ satisfies the $J^2$-condition, i.e., for any orthogonal vectors $Z_1, Z_2\in\frak{z}$,
there exists $Z_3\in\frak{z}$ such that $J_{Z_1}\circ J_{Z_2}=J_{Z_3}$,
then we can find that
$$
J_{Z^\perp} V=J_{Z^\perp} (J_Z V),\quad
\ker(\ad(V))_0=\ker(\ad(J_Z V))_0.
$$

\subsection{The operator $K$ and the map $\overline{K}$}\label{secK}

In this section, we choose arbitrary unit vectors $V\in\frak{v}, Z\in\frak{z}$ and fix them.

\begin{Def}\label{defK}
We define an endomorphism $K=K_{V, Z} : Z^\perp\rightarrow Z^\perp$ by
$$
K(X):=\ad(V)\circ \ad(J_ZV)^\ast (X)= [V, J_X J_Z V].
$$
\end{Def}

We can immediately see that $K_{V,Z}$ is skew-symmetric with respect to $\langle\,\cdot\,,\,\cdot\,\rangle_{\mathfrak{n}}$.

The operator $K$ can be interpreted as follows.
From the decomposition \eqref{decomp1} of $\frak{v}$,
we find that for $X\in Z^\perp$ there exist $U\in \ker(\ad(V))_0$ and $Z'\in Z^\perp$ such that $J_X J_Z V=U+J_{Z'} V$.
Then, by using lemma \ref{gHprop} (viii), we have
$$
K(X)=[V, J_X J_Z V]=[V, J_{Z'} V+U]=[V, J_{Z'} V]=|V|^2 Z'=Z',
$$
from which we have
\begin{equation}\label{Kmean}
J_X J_Z V=J_{K_{V,Z}(X)} V+U.
\end{equation}
Therefore, $J_{K_{V,Z}(X)} V$ is the $J_{\frak{z}}V$-part of $J_X J_Z V$.


\begin{rem}
Since $K^2$ is a symmetric operator, the eigenvalues of $K^2$ are all real numbers and it is also known that they lie within $[-1, 0]$.
Moreover, the eigenspace $L_\mu$ corresponding to a non-zero eigenvalue $\mu$ has even dimension,
because if $K^2(X)=\mu X$, then $K^2(K(X))=\mu K(X)$.
See \cite[pp.33--34]{BTV} for details.
\end{rem}

\begin{Def}\label{defKbar}
We define a linear map $\overline{K}_{V, Z} : \ker(\ad(V))_0\rightarrow Z^\perp$ by the restriction of $\ad(J_Z V)$ to $\ker(\ad(V))_0$.
\end{Def}

\begin{lem}
The operator $K=K_{V, Z}$ and the map $\overline{K}=\overline{K}_{V, Z}$ satisfy
\begin{equation}\label{KKbar}
\overline{K}\circ \overline{K}^*=\mathrm{id}_{Z^\perp}+K^2,
\end{equation}
where the asterisk $^*$ means the formal adjoint map with respect to $\langle\,\cdot\,,\,\cdot\,\rangle_\mathfrak{n}$.
\end{lem}

\begin{proof}
At first, we remark that $\ad(J_Z V)^*(Z^\perp)$ is not a subset of $\ker(\ad(V))_0$.
Hence, we have $\overline{K}^*=\pi_0\circ \ad(J_Z V)^*$,
where $\pi_0$ is the projection from $\frak{v}$ onto $\ker(\ad(V))_0$.
From \eqref{Kmean}, we have for any $X\in Z^\perp$
\begin{align*}
\overline{K}\circ \overline{K}^*(X)
=&\ad(J_Z V)\circ \pi_0\circ \ad(J_Z V)^*(X)\\
=&\ad(J_Z V)\circ \pi_0(J_X J_Z V)\\
=&\ad(J_Z V)(J_XJ_Z V-J_{K_{V,Z(X)}}V)\\
=&[J_Z V,J_X J_Z V]-[J_Z V,J_{K_{V,Z(X)}}V]\\
=&|J_Z V|^2X-[J_{K_{V,Z(X)}}J_Z V,V]\\
=&X+[V, J_{K_{V,Z}(X)}J_Z V]=X+K_{V,Z}^2(X),
\end{align*}
from which \eqref{KKbar} is obtained.
\end{proof}

\begin{rem}
\begin{enumerate}
\item the subspace $\ker(\ad(V))_0\cap \ker(\ad(J_Z V))_0$ in \eqref{decomp3} is the kernel of $\overline{K}_{V, Z}$.
\item If $\frak{n}$ satisfies the $J^2$-condition, then $\overline{K}_{V, Z}^\ast=0$.
Hence, in this case, $K_{V, Z}^2=-\mathrm{id}_{Z^\perp}$.
\end{enumerate}
\end{rem}

\subsection{Damek-Ricci spaces}

Let $(\frak{n}, [\cdot,\cdot]_\frak{n}, \langle\cdot,\cdot\rangle_\frak{n})$
be a generalized Heisenberg algebra.
Let $\frak{s}:=\frak{n}\oplus \frak{a}$,
where $\frak{a}=\mathbb{R} A$ is a one dimensional real vector space with generator $A$,
and define a bracket product $[\cdot,\cdot]_\frak{s}$ and a inner product
$\langle\cdot,\cdot\rangle_\frak{s}$ by
\begin{align}
\label{DR_Lbs}
[U_1+Y_1+s_1A, U_2+Y_2+s_2A]_\frak{s}
=&\frac{s_1}{2}U_2-\frac{s_2}{2}U_1+s_1Y_2-s_2Y_1+[U_1,U_2]_\frak{n},\\
\label{DR_Rmtrc}
\langle U_1+Y_1+s_1A, U_2+Y_2+s_2A\rangle_\frak{s}
=&\langle U_1, U_2\rangle_\frak{n}+\langle Y_1, Y_2\rangle_\frak{n}+s_1s_2,
\end{align}
where $U_i\in\frak{v}, Z_i\in\frak{z}, s_i\in\mathbb{R}$ ($i=1, 2$),
respectively.
We find immediately that the derived subalgebra
$[\frak{s},\frak{s}]_\frak{s}$ of $\frak{s}$ is equal to $\frak{n}$.
This shows that $\frak{s}$ is a solvable Lie algebra.

\begin{Def}[\cite{DR}]
Let $S$ be a simply connected Lie group whose Lie algebra is $(\frak{s}, [\,\cdot\,,\,\cdot\,]_{\frak{s}})$ equipped with the left invariant metric $g$ induced by $\langle\cdot,\cdot\rangle_\frak{s}$.
We call $(S, g)$ a \textit{Damek-Ricci space}.
\end{Def}

If we identify $S$ with $\frak{v}\times\frak{z}\times\mathbb{R}_+$
via the exponential map, the group structure on $S$ is given by
\begin{equation}\label{DR_m}
(V_1,Z_1,a_1)\cdot(V_2,Z_2,a_2)
=\left(V+\sqrt{a_1}V_2,\,Z+a_1Z_2+\frac{\sqrt{a_1}}{2}[V_1, V_2],\,a_1 a_2\right).
\end{equation}
We refer to \cite[Chap. 4]{BTV} for details.

\subsection{The Busemann function}

Let $(M, g)$ be a complete, simply connected, non-compact Riemannian manifold with no focal points.
Here no focal point means that every geodesic has no focal point as a one dimensional submanifold (see \cite{I}).
Then, for a geodesic $\gamma : \mathbb{R}\rightarrow M$ with unit speed,
we can define a function $b_\gamma$ on $M$, called the \textit{Busemann function}, by
\begin{equation}
b_\gamma(x)=\lim_{t\rightarrow \infty}(d(\gamma(t), x)-t).
\end{equation}
The Busemann function is a $C^2$--convex function \cite{BGS}.
Then we can define the Hessian $\nabla d b_\gamma$ of the Busemann function and find that $\nabla d b_\gamma$ is positive semi-definite.

\begin{rem}
The gradient vector field $\nabla b_\gamma$ of the Busemann function $b_\gamma$ is a unit vector field.
Hence, we find that the Hessian of $b_\gamma$ has $0$-eigenvalue associated with $\nabla b_\gamma$.
\end{rem}

Every Riemannian manifold $M$, such as those described above, carries its ideal boundary, denoted by $\partial M$, which is a quotient space of all geodesic rays on $M$ divided by an equivalence relation $\sim$ defined as follows:
$\gamma_1\sim\gamma_2$ if and only if $d(\gamma_1(t), \gamma_2(t))<\infty$ on $t\in [0, \infty)$.
If $\gamma_1\sim\gamma_2$, then $b_{\gamma_1}-b_{\gamma_2}$ is a constant function on $M$.
Hence, we find that if $\gamma_1\sim\gamma_2$, then the Hessian of $b_{\gamma_1}$ coincides with of $b_{\gamma_2}$.
Namely, the Hessian of the Busemann function, denoted by $H_\theta$, is determined for a boundary point $\theta\in\partial M$.

We refer to \cite{I} for details about the Busemann function.
\smallskip

A Damek-Ricci space $S$ is a Hadamard manifold, i.e., a complete, simply connected, non-compact Riemannian manifold whose sectional curvature is non-positive.
It is known that such a space has no focal points (see \cite{I}).
The ideal boundary of $S$ is identified with $N\cup\{\infty\}$ (see \cite{ADY, IS2010}).
Let $b_\theta$ be the Busemann function of a Damek-Ricci space $S$.
We refer to a group-theoretic representation of the Busemann function to \cite{IS2010} as follows.

\begin{prop}[{\cite[Theorem 4]{IS2010}}]\label{busemannDR}
For $p=(V,Z,a)\in S\simeq \frak{v}\times\frak{z}\times \mathbb{R}$,
the Busemann function is represented by
\begin{equation}
b_\theta(p)=\left\{
\begin{aligned}
&\log\left(
\frac{\left(
a+\frac{1}{4}|v-V|_{\mathfrak{n}}^2
\right)^2
+\left|
z-Z-\frac{1}{2}[V,v]_{\frak{n}}
\right|^2}{a
\left(
\left(1+\frac{1}{4}|v|_{\mathfrak{n}}^2\right)^2+|z|_{\mathfrak{n}}^2
\right)}
\right),&\theta&=(v,z)\\
&-\log a,&\theta&=\infty
\end{aligned}
\right.
\end{equation}
\end{prop}

\section{Hessian of the Busemann function on $S$}


In this section, we compute the Hessian $H_\theta=\nabla db_\theta$ of the Busemann function $b_\theta$ on a Damek-Ricci space $S$.
We consider both cases: $\theta=\infty$ and $\theta\ne\infty$, and analyze the eigenvalues of $H_\theta$ in detail.
The map $K^2$ and its spectral properties play a key role in our computations.

\subsection{Components of $H_\theta$}

To proceed, we introduce a global coordinate system for $S\simeq \frak{v}\times\frak{z}\times\mathbb{R}_+$.
Let $\{e_i\}_{i=1,\ldots,k}$ and $\{e_{k+r}\}_{r=1,\ldots,m}$ be orthonormal bases of $\frak{v}$ and $\frak{z}$, respectively, with $V^i, i=1,\ldots, k$, and $Z^r, r=1,\ldots, m$ as the corresponding coordinates on $N\simeq\frak{v}\times\frak{z}$.
Here $k=\dim \frak{v}$ and $m=\dim \frak{z}$.
Moreover, we set $a : \mathbb{R}A\ni sA\mapsto e^s\in \mathbb{R}_+$.
Then, a left-invariant orthonormal frame field of $S$ at $p=(V, Z, a)$ is written as
\begin{equation}\label{ONF_S}
\left\{
\begin{aligned}
E_0=&\,a\frac{\partial}{\partial a},&&\\
E_i=&\,\sqrt{a}\frac{\partial}{\partial V^i}
-\frac{\sqrt{a}}{2}\sum_{j=1}^k\sum_{r=1}^mA_{ij}^rV^j\frac{\partial}{\partial Z^r},&(i=&1,2,\ldots,k),\\
E_{k+r}=&\,a\frac{\partial}{\partial Z^r},&(r=&1,2,\ldots,m),
\end{aligned}
\right.
\end{equation}
where $A_{ij}^r:=\langle[e_i,e_i]_{\frak{n}},e_{k+r}\rangle$ (see \cite[4.1.5]{BTV}).
Then, we have
\begin{gather}
\nabla_{E_i}E_j
=\frac{1}{2}\sum_r A_{ij}^r E_{k+r}+\frac{1}{2}\delta_{ij}E_0,\quad
\nabla_{E_i}E_{k+r}=\nabla_{E_{k+r}}E_i
=-\frac{1}{2}\sum_jA_{ij}^r E_j,\notag\\
\nabla_{E_i}E_0=-\frac{1}{2}E_i,\quad
\nabla_{E_{k+r}}E_{k+l}=\delta_{rl}E_0,\quad
\nabla_{E_{k+r}}E_0=-E_{k+r},\quad\notag
\nabla_{E_0}E_\alpha=0,\notag
\end{gather}
where $1\le i, j\le k$, $1\le r, l\le m$ and $1\le \alpha\le k+m$ (see \cite[4.1.6]{BTV}).

Now we compute components of $H_\theta=\left(b_{\alpha,\beta}\right)$ with respect to $\{E_\alpha\}_{\alpha=0, 1, \ldots, k+m}$,
i.e.,
\begin{equation*}
b_{\alpha,\beta}
:=(\nabla_{E_\alpha} db_\theta)(E_\beta)
=E_\alpha(E_\beta b_\theta)-(\nabla_{E_\alpha} E_\beta)b_\theta.
\end{equation*}

\subsubsection{The case of $\theta=\infty$}

From Proposition \ref{busemannDR}, we immediately obtain the following.

\begin{thm}\label{H_infty}
If $\theta=\infty$, then we have
\begin{equation*}
b_{i,j}=\frac{1}{2}\delta_{ij},\qquad
b_{k+r,k+l}=\delta_{rl},\qquad
\mbox{(otherwise)}=0.
\end{equation*}
\end{thm}

From this theorem, we obtain Theorem \ref{main0} (1).
This result aligns with the known structure of horospheres in symmetric spaces.

\subsubsection{The case of $\theta\ne\infty$}

To compute $H_\theta$ in the case $\theta\ne \infty$, it is convenient to define certain maps on $S\times N$ by
\begin{equation}\label{vec}
\begin{split}
\mathcal{V} : S\times N\rightarrow \frak{v};\ \mathcal{V}(p, \theta)&=v-V,\\
\mathcal{Z} : S\times N\rightarrow \frak{z};\ \mathcal{Z}(p, \theta)&=z-Z-\frac{1}{2}[V,v]_{\frak{n}},
\end{split}
\end{equation}
and positive functions $f, F : S\times N \rightarrow \mathbb{R}_+$
\begin{equation}\label{func}
f(p, \theta)=a+\frac{1}{4}|\mathcal{V}(p, \theta)|^2,\qquad
F(p, \theta)=f(p, \theta)^2+|\mathcal{Z}(p, \theta)|^2,
\end{equation}
where $x=(V, Z, a)\in S$ and $\theta=(v, z)\in N\simeq \frak{v}\times\frak{z}$.
Then, we have 
\begin{equation}\label{BusemannDR}
b_\theta(p)=\log F(p,\theta)-\log a+C(\theta).
\end{equation}
Here $N$ means $\partial S\backslash\{\infty\}$ and $C(\theta)$ is a certain function on $N$.
Below, we will omit the variables $(p, \theta)$ and denote simply as $\mathcal{V}, \mathcal{Z}, f$ and $F$.
\smallskip

Easy computation shows that
\begin{align*}
E_0 f=&a,&
E_i f=&-\frac{\sqrt{a}}{2}\langle\mathcal{V}, e_i\rangle,&
E_{k+r} f=&0,\\
E_0F=&2af,&
E_iF=&-\sqrt{a}\langle f\mathcal{V}-J_\mathcal{Z}\mathcal{V}, e_i\rangle,&
E_{k+r}F=&-2a\langle \mathcal{Z},e_{k+r}\rangle,
\end{align*}
from which, by using lemma \ref{gHprop}, we have
\begin{align}
\label{comp00}
b_{0,0}
=&\frac{2a}{F^2}\left(fF+aF-2af^2\right),\\
\label{comp0i}
b_{0,i}
=&-\frac{\sqrt{a}}{2F^2}\left\{
\left(fF+2aF-4af^2\right)\langle\mathcal{V}, e_i\rangle
+\left(4af-F\right)\langle J_\mathcal{Z}\mathcal{V}, e_i\rangle
\right\},\\
\label{comp0r}
b_{0,k+r}
=&\frac{2a}{F^2}\left(2af-F\right)\langle \mathcal{Z},e_{k+r}\rangle,\\
b_{i,j}
=&\frac{a}{2F}\left(
\langle \mathcal{V}, e_i\rangle \langle \mathcal{V}, e_j\rangle
+\langle [e_i, \mathcal{V}]_\frak{n}, [e_j, \mathcal{V}]_\frak{n}\rangle\right)\nonumber\\
&-\frac{a}{F^2}
\langle f\mathcal{V}-J_\mathcal{Z}\mathcal{V}, e_i\rangle
\langle f\mathcal{V}-J_\mathcal{Z}\mathcal{V}, e_j\rangle
+\frac{1}{2}\delta_{ij},\label{compij}\\
b_{i,k+r}
=&-\frac{2a\sqrt{a}}{F^2}
\langle f\mathcal{V}-J_\mathcal{Z}\mathcal{V}, e_i\rangle
\langle \mathcal{Z}, e_{k+r}\rangle\nonumber\\
&-\frac{\sqrt{a}}{2F}\langle [e_i, (f-2a)\mathcal{V}-J_\mathcal{Z}\mathcal{V}]_\frak{n},e_{k+r}\rangle,\label{compir}\\
\label{comprl}
b_{k+r,k+l}
=&-\frac{4a^2}{F^2}\langle\mathcal{Z},e_{k+r}\rangle\langle\mathcal{Z},e_{k+l}\rangle
+\frac{1}{F}\left(F-2af+2a^2\right)\delta_{rl}.
\end{align}

\begin{prop}\label{H_neinfty}
Assume that $\theta\ne \infty$.
\begin{enumerate}
\item[(I)] If $\mathcal{V}=0$ and $\mathcal{Z}=0$, then we have
\begin{equation*}
b_{i,j}=\frac{1}{2}\delta_{ij},\quad
b_{k+r,k+l}=\delta_{rl},\quad
\mbox{(otherwise)}=0.
\end{equation*}
\item[(II)] If $\mathcal{V}=0$ and $\mathcal{Z}\ne 0$, then we have
$$
b_{0,0}
=\frac{4a^2}{F^2}\left(F-a^2\right),\quad
b_{i,j}=\frac{1}{2}\delta_{ij},\quad
b_{k+r,k+l}=
-\frac{4a^2}{F^2}\langle\mathcal{Z},e_{k+r}\rangle\langle\mathcal{Z},e_{k+l}\rangle
+\delta_{rl},
$$
$$
b_{0,k+r}=
\frac{2a}{F^2}\left(2a^2-F\right)\langle\mathcal{Z}, e_{k+r}\rangle,\quad
\mbox{(otherwise)}=0.
$$
\item[(III)] If $\mathcal{V}\ne0$ and $\mathcal{Z}=0$, then we have
$$
b_{0,0}=\frac{2a(f-a)}{f^2},\quad
b_{i,j}=\frac{a}{2f^2}\left(
-\langle\mathcal{V}, e_i\rangle\langle\mathcal{V}, e_j\rangle
+\langle [e_i, \mathcal{V}]_\frak{n}, [e_j, \mathcal{V}]_\frak{n}\rangle\right)
+\frac{1}{2}\delta_{ij},\\
$$
$$
b_{0,i}=-\frac{\sqrt{a}(f-2a)}{2f^2}\langle\mathcal{V},e_i\rangle,\quad
b_{i,k+r}=\frac{\sqrt{a}(f-2a)}{2f^2}\langle e_i,J_{e_{k+r}}\mathcal{V}\rangle,
$$
$$
b_{k+r,k+l}=\frac{1}{F}\left(F-2af+2a^2\right)\delta_{rl},\quad
\mbox{(otherwise)}=0.
$$
\end{enumerate}
\end{prop}

\begin{proof}
From \eqref{func}, if $\mathcal{V}=0$ and $\mathcal{Z}=0$, then we have $f=a$ and $F=f^2$, respectively.
Substituting these equations into \eqref{comp00} -- \eqref{comprl}, we obtain our lemma.
\end{proof}

From the above proposition, we obtain Theorem \ref{main0} (2-i)

\begin{proof}[Proof of Theorem \ref{main0} (2-i)]
\underline{(I)\ Case of $\mathcal{V}=0$ and $\mathcal{Z}=0$} :
Our assertion is trivial from lemma \ref{H_neinfty} (i).
\smallskip

\noindent\underline{(II)\ Case of $\mathcal{V}=0$ and $\mathcal{Z}\ne0$} :
We choose an orthonormal basis $\{e_{k+r}\}_{r=1,\ldots,m}$ of $\frak{z}$ satisfying
\begin{equation}\label{onb_V}
e_{k+1}=\overline{\mathcal{Z}}
=\frac{1}{|\mathcal{Z}|}\mathcal{Z},\quad
\mbox{and}\quad
e_{k+r}\in \mathcal{Z}^\perp, r=2,\ldots,m.
\end{equation}
Then, we can write the matrix $(b_{\alpha,\beta})$,
of which some rows and columns are interchanged, as
\begin{equation*}
\left(
\begin{array}{ccc}
B_1&O&O\\
O&\frac{1}{2}\mathrm{id}_{\frak{v}}&O\\
O&O&\mathrm{id}_{\mathcal{Z}^\perp}
\end{array}
\right),
\end{equation*}
where
\begin{align*}
B_1=&\left(
\begin{array}{cc}b_{0,0}&b_{0,k+1}\\b_{k+1,0}&b_{k+1,k+1}
\end{array}
\right)\\
=&\frac{1}{F^2}\left(
\begin{array}{cc}
4a^2\left(F-a^2\right)&
2a\left(2a^2-F\right)\sqrt{F-a^2}\\
2a\left(2a^2-F\right)\sqrt{F-a^2}&
\left(F-2a^2\right)^2\\
\end{array}
\right).
\end{align*}
Easy computation show that $\det(B_1)=0$ and $\det\left(B_1-I_2\right)=0$
from which we find that eigenvalues of the $2\times 2$-matrix $B_1$ are $0$ and $1$.
Hence, in this case the matrix $(b_{\alpha, \beta})$ has eigenvalues $0, \frac{1}{2}$ and $1$ whose multiplicities are $1$, $k$ and $1+(m-1)=m$, respectively.
\smallskip

\noindent\underline{(III)\ Case $\mathcal{V}\ne 0$ and $\mathcal{Z}=0$} :
We choose an orthonormal basis of $\frak{v}$ satisfying
\begin{equation}\label{onb_Z}
\left\{
\begin{aligned}
&e_{1}=\overline{\mathcal{V}}=\frac{1}{|\mathcal{V}|}\mathcal{V},&&\\
&e_i\in\ker(\ad(\mathcal{V}))_0,&i=&2,\ldots,k-m,\\
&e_{(k-m)+r}=J_{e_{k+r}}\overline{\mathcal{V}}=\ad(\overline{\mathcal{V}})^\ast(e_{k+r})\in J_\frak{z}\mathcal{V},&r=&1,\ldots,m,
\end{aligned}
\right.
\end{equation}
We remark that in this case $e_{k+r}=\ad(\overline{\mathcal{V}})(e_{(k-m)+r})$.
Then, we can write the matrix $H_\theta=(b_{\alpha,\beta})$ of which some rows and columns are interchanged, as
\begin{equation*}
\left(
\begin{array}{ccc}
B_2&O&O\\
O&\frac{1}{2}\id_{\ker(\ad(\mathcal{V}))_0}&O\\
O&O&B_3
\end{array}
\right),
\end{equation*}
where
\begin{equation*}
\begin{split}
B_2=&\left(
\begin{array}{cc}b_{0,0}&b_{0,1}\\b_{1,0}&b_{1,1}
\end{array}
\right)\\
=&\frac{1}{f^2}\left(
\begin{array}{cc}
2a(f-a)&-\sqrt{a(f-a)}(f-2a)\\
-\sqrt{a(f-a)}(f-2a)&\frac{1}{2}\left(f-2a\right)^2
\end{array}
\right),\\
B_3=&\left(
\begin{array}{cc}
(b_{k-m+r,k-m+l})&
(b_{k-m+r,k+l})\\
(b_{k+r,k-m+l})&
(b_{k+r,k+l})
\end{array}
\right)\\
=&\frac{1}{f^2}\left(
\begin{array}{cc}
\left\{2a(f-a)+\frac{1}{2}f^2\right\}\id_{J_\frak{z}\mathcal{V}}&
-\sqrt{a(f-a)}(f-2a)\ad(\overline{\mathcal{V}})^\ast\circ\id_{\overline{\mathcal{Z}}^\perp}\\
-\sqrt{a(f-a)}(f-2a)\mathrm{id}_{\mathcal{Z}^\perp}\circ\ad(\overline{\mathcal{V}})&
\left(f^2-2af+2a^2\right)\id_\frak{z}
\end{array}
\right).
\end{split}
\end{equation*}
Easy computation show that $\det(B_2)=0$ and $\det\left(B_2-\frac{1}{2}I_2\right)=0$ from which we find that eigenvalues of the $2\times 2$-matrix $B_2$ are $0$ and $\frac{1}{2}$.
On the other hand, by appropriately permuting the rows and columns of $B_3$,
it becomes a block diagonal matrix with the following $2\times 2$-matrix
$$
B_3'=\frac{1}{f^2}\left(
\begin{array}{cc}
2a(f-a)+\frac{1}{2}f^2&
-\sqrt{a(f-a)}(f-2a)\\
-\sqrt{a(f-a)}(f-2a)&
f^2-2a(f-a)
\end{array}
\right).
$$
It is easy to check that satisfies $B_3'$ satisfies $\det(B'_3-I_2)=\det(B'_3-\frac{1}{2}I_2)=0$,
which means that $B_3$ has eigenvalues $\frac{1}{2}$ and $1$ whose multiplicities are $m$.
Hence, in this case the matrix $(b_{\alpha, \beta})$ has eigenvalues $0, \frac{1}{2}$ and $1$ whose multiplicities are $1$, $1+(k-m-1)+m=k$ and $m$, respectively.
\end{proof}

From these computations, we find that when $\mathcal{V}=0$ or $\mathcal{Z}=0$,
the eigenvalues remain $0, \dfrac{1}{2}$, and $1$.
For the general case, the eigenvalues are solutions of a cubic equation, indicating a richer structure.

\subsubsection{The case of $\theta\ne\infty$, $\mathcal{V}\ne 0$ and $\mathcal{Z}\ne 0$}

We choose orthonormal basis $\{e_i, e_{k+r}\}_{1\le i\le k, 1\le r\le m}$ of $\frak{n}=\frak{v}\oplus\frak{z}$ satisfying \eqref{onb_V} and \eqref{onb_Z}.
\bigskip

We set a $4\times 4$-matrix  $B_0$ defined by
\begin{equation*}
B_0
=\left(\begin{array}{cccc}
b_{0,0}&b_{0,1}&b_{0,k-m+1}&b_{0,k+1}\\
b_{1,0}&b_{1,1}&b_{1,k-m+1}&b_{1,k+1}\\
b_{k-m+1,0}&b_{k-m+1,1}&b_{k-m+1,k-m+1}&b_{k-m+1,k+1}\\
b_{k+1,0}&b_{k+1,1}&b_{k+1,k-m+1}&b_{k+1,k+1}\\
\end{array}\right),
\end{equation*}
where
\begin{align*}
b_{0,0}
=&\frac{2a}{F^2}\left(fF+aF-2af^2\right),\\
b_{1,0}
=-b_{k+1,k-m+1}
=&-\frac{1}{F^2}\left(fF+2aF-4af^2\right)\sqrt{a(f-a)},\\
b_{k-m+1,0}
=b_{k+1,1}
=&-\frac{1}{F^2}(4af-F)\sqrt{a(f-a)(F-f^2)},\\
b_{k+1,0}
=&\frac{2a}{F^2}\left(2af-F\right)\sqrt{F-f^2},\\
b_{1,1}
=&-\frac{2a}{F^2}(f-a)(2f^2-F)+\frac{1}{2},\\
b_{k-m+1,1}
=&\frac{4af}{F^2}(f-a)\sqrt{F-f^2},\\
b_{k-m+1,k-m+1}
=&\frac{1}{2}+\frac{2a}{F^2}(f-a)(2f^2-F),\\
b_{k+1,k+1}
=&1-\frac{2a}{F^2}\left(fF+aF-2af^2\right).
\end{align*}

\begin{thm}\label{B1posi}
The eigenvalues of $B_0$ are $0, \frac{1}{2}$ and $1$ whose multiplicities are $1, 2$ and $1$, respectively.
\end{thm}

\begin{proof}
We can compute that
$$\det(B_0)=\det\left(B_0-I_4\right)=\det\left(B_0-\frac{1}{2}I_4\right)=0
\quad\mbox{and}
\quad
\mathrm{tr}(B_0)=2,
$$
from which we obtain our result.
Here $I_n$ is the $n\times n$ identity matrix.
\end{proof}

The above theorem implies Theorem \ref{main0} (2-ii-a).
\smallskip

From \eqref{compir}, for $2\le i\le k-m$ and $2\le r\le m$, we have
\begin{equation*}\label{B2compo}
b_{i,k+r}
=\frac{\sqrt{a}}{2F}\langle [e_i, J_\mathcal{Z} \mathcal{V}]_{\mathfrak{n}}, e_{k+r}\rangle
=-\dfrac{\sqrt{a}}{2F}|\mathcal{V}||\mathcal{Z}|
\langle \overline{K}_{\overline{\mathcal{V}},\overline{\mathcal{Z}}}(e_i), e_{k+r}\rangle.
\end{equation*}
On the other hand, from \eqref{compir}, for $2\le r, l\le m$,
\begin{align*}
b_{(k-m)+l,k+r}
=&-\frac{\sqrt{a}}{2F}\langle [J_{e_{k+l}}\overline{\mathcal{V}}, (f-2a)\mathcal{V}-J_\mathcal{Z}\mathcal{V}]_\frak{n},e_{k+r}\rangle\\
=&-\frac{\sqrt{a}}{2F}\left\{
(f-2a)\langle [J_{e_{k+l}}\overline{\mathcal{V}}, \mathcal{V}]_\frak{n},e_{k+r}\rangle
-\langle [J_{e_{k+l}}\overline{\mathcal{V}}, J_\mathcal{Z}\mathcal{V}]_\frak{n},e_{k+r}\rangle
\right\}\\
=&-\frac{\sqrt{a}}{2F}\left\{
-(f-2a)\langle\overline{\mathcal{V}}, \mathcal{V}\rangle \langle e_{k+l}, e_{k+r}\rangle
-\langle [\overline{\mathcal{V}}, J_{e_{k+l}}J_\mathcal{Z}\mathcal{V}]_\frak{n},e_{k+r}\rangle
\right\}\\
=&\frac{\sqrt{a}}{2F}\left\{
(f-2a) |\mathcal{V}| \rangle \langle e_{k+l}, e_{k+r}\rangle+ |\mathcal{V}| |\mathcal{Z}|\langle \overline{K}_{\overline{\mathcal{V}},\overline{\mathcal{Z}}}(e_{k+l}), e_{k+r}\rangle
\right\}.
\end{align*}

%

\begin{prop}\label{Bblemma}
We set 
\begin{align}
\label{b1}
b_1=&\frac{1}{2}+\frac{2a}{F}(f-a)=\frac{1}{2}+\frac{a}{2F}|\mathcal{V}|^2,\\
\label{b2}
b_2=&1-\frac{2a}{F}(f-a)=1-\frac{a}{2F}|\mathcal{V}|^2,\\
\label{b3}
b_3=&\frac{1}{F}(f-2a)\sqrt{a(f-a)}=\frac{1}{2F}(f-2a)\sqrt{a}|\mathcal{V}|,\\
\label{b4}
b_4=&\frac{\sqrt{a}}{2F}|\mathcal{Z}|\,|\mathcal{V}|.
\end{align}
Then,  the matrix $H_\theta=(b_{\alpha, \beta})$,
of which some rows and columns are interchanged can be described as follows:
\begin{equation*}\label{matBB}
H_\theta=\left(\begin{array}{cc}
B_0&O\\
O&B\\
\end{array}\right),
\end{equation*}
where
\begin{equation}\label{mat}
B=\left(\begin{array}{ccc}
\frac{1}{2}\mathrm{id}_{\ker(\ad(\overline{\mathcal{V}}))_0}&
O&
-b_4 \overline{K}_{\overline{\mathcal{V}},\overline{\mathcal{Z}}}^*
\\
O&
b_1\mathrm{id}_{J_{\overline{\mathcal{Z}}^\perp}\overline{\mathcal{V}}}
&
\ad(\overline{\mathcal{V}})^\ast
\circ
\left(b_3\,\mathrm{id}_{\overline{\mathcal{Z}}^\perp}-b_4 K_{\overline{\mathcal{V}},\overline{\mathcal{Z}}}\right)
\\
-b_4 \overline{K}_{\overline{\mathcal{V}},\overline{\mathcal{Z}}}&
\left(b_3\,\mathrm{id}_{\mathcal{Z}^\perp}+b_4K_{\overline{\mathcal{V}}, \overline{\mathcal{Z}}}\right)\circ
\ad(\overline{\mathcal{V}})
&
b_2\,\mathrm{id}_{\overline{\mathcal{Z}}^\perp}
\end{array}\right).
\end{equation}
Moreover, $b_i, i=1,2,3,4$ satisfy
\begin{equation}\label{bseq}
b_1b_2-b_3^2-b_4^2=\frac{1}{2},\quad
b_1+b_2=\dfrac{3}{2}.
\end{equation}
\end{prop}

Here, we examine the eigenvalues and eigenvectors of $B$.
Let 
$$\bm{x}=U+J_{Z_1}\overline{\mathcal{V}}+Z_2
=\left(\begin{array}{c}U\\J_{Z_1}\overline{\mathcal{V}}\\Z_2\end{array}\right)\in T_pS$$
be an eigenvector of $B$ corresponding to the eigenvalue $\lambda$, where $U\in\ker(\ad(\mathcal{V}))_0$, $Z_1, Z_2\in \mathcal{Z}^\perp$.
The equation $(B-\lambda I)\bm{x}=0$ is equivalent to the following three equations:
\begin{align}
\label{B-I-1}
\left(\frac{1}{2}-\lambda\right)U-b_4\,\overline{K}_{\overline{\mathcal{V}},\overline{\mathcal{Z}}}^*Z_2=0,\\
\label{B-I-2}
\left(b_1-\lambda\right) J_{Z_1}\overline{\mathcal{V}}
+\ad(\overline{\mathcal{V}})^\ast
\circ
\left(b_3\,\mathrm{id}_{\overline{\mathcal{Z}}^\perp}-b_4 K_{\overline{\mathcal{V}},\overline{\mathcal{Z}}}\right)Z_2
=0,\\
\label{B-I-3}
-b_4 \overline{K}_{\overline{\mathcal{V}},\overline{\mathcal{Z}}}U
+\left(b_3\,\mathrm{id}_{\mathcal{Z}^\perp}+b_4K_{\overline{\mathcal{V}}, \overline{\mathcal{Z}}}\right)Z_1
+\left(b_2-\lambda\right)Z_2=0.
\end{align}
From these equations, we derive the conditions that $U, Z_1$, and $Z_2$ must satisfy.

By applying $\mathrm{ad(\overline{\mathcal{V}})}$ to both sides of \eqref{B-I-2},
we have
\begin{equation}\label{B-I-2'}
(b_1-\lambda) Z_1
+\left(b_3\,\mathrm{id}_{\overline{\mathcal{Z}}^\perp}-b_4 K_{\overline{\mathcal{V}},\overline{\mathcal{Z}}}\right)Z_2
=0,
\end{equation}
which means that $Z_1$ is expressed by using $Z_2$.
By multiplying $\left(\frac{1}{2}-\lambda\right)(b_1-\lambda)$ to \eqref{B-I-3} and eliminating $U$ and $Z_1$ using \eqref{B-I-1} and \eqref{B-I-2'}, we obtain
\begin{align*}
0
=&
-\left(b_1-\lambda\right)b_4^2 \overline{K}_{\overline{\mathcal{V}},\overline{\mathcal{Z}}}\overline{K}_{\overline{\mathcal{V}},\overline{\mathcal{Z}}}^\ast Z_2\\
&-\left(\frac{1}{2}-\lambda\right)\left(b_3\,\mathrm{id}_{\mathcal{Z}^\perp}+b_4 K_{\overline{\mathcal{V}},\overline{\mathcal{Z}}}\right)
\left(b_3\,\mathrm{id}_{\mathcal{Z}^\perp}-b_4 K_{\overline{\mathcal{V}},\overline{\mathcal{Z}}}\right)Z_2\\
&+\left(\frac{1}{2}-\lambda\right)\left(b_1-\lambda\right)\left(b_2-\lambda\right)Z_2\nonumber\\
=&
-\left(b_1-\lambda\right)b_4^2 \left(\mathrm{id}_{\mathcal{Z}^\perp}+K_{\overline{\mathcal{V}},\overline{\mathcal{Z}}}^2\right)Z_2
-\left(\frac{1}{2}-\lambda\right)
\left(b_3^2\mathrm{id}_{\mathcal{Z}^\perp}-b_4^2K_{\overline{\mathcal{V}},\overline{\mathcal{Z}}}^2\right)Z_2\\
&+\left(\frac{1}{2}-\lambda\right)\left(\lambda^2-\dfrac{3}{2}\lambda+b_1b_2\right)Z_2\nonumber\\
=&
\left\{
-\left(b_1-\lambda\right)+\left(\frac{1}{2}-\lambda\right)
\right\}b_4^2
\left(\mathrm{id}_{\mathcal{Z}^\perp}+K_{\overline{\mathcal{V}},\overline{\mathcal{Z}}}^2\right)Z_2\\
&+\left(\frac{1}{2}-\lambda\right)
\left(\lambda^2-\frac{3}{2}\lambda+b_1b_2-b_3^2-b_4^2\right)Z_2\\
=&\left(\frac{1}{2}-b_1\right)b_4^2
\left(\mathrm{id}_{\mathcal{Z}^\perp}+K_{\overline{\mathcal{V}},\overline{\mathcal{Z}}}^2\right)Z_2
+\left(\frac{1}{2}-\lambda\right)^2\left(1-\lambda\right)Z_2,
\end{align*}
i.e., 
\begin{align}
K_{\overline{\mathcal{V}},\overline{\mathcal{Z}}}^2 Z_2
=&-\left(
1-\dfrac{\left(\frac{1}{2}-\lambda\right)^2 (1-\lambda)}{b_4^2(b_1-\frac{1}{2})}
\right)Z_2\notag\\
=&-\left\{1+\frac{8F^3}{a^2|\mathcal{V}|^4|\mathcal{Z}|^2}\left(\lambda-\dfrac{1}{2}\right)^2(\lambda-1)\right\}Z_2,\label{BengK2eng}
\end{align}
from which we obtain Theorem \ref{main0} (2-ii-b).
From the above, the following can be concluded about $Z_1, Z_2\in Z^\perp$;
\begin{itemize}
\item[(i)] $Z_2\in L_{\mu}$,
where
\begin{equation}\label{mu}
\mu=-\left\{1+\frac{8F^3}{a^2|\mathcal{V}|^4|\mathcal{Z}|^2}\left(\lambda-\dfrac{1}{2}\right)^2(\lambda-1)\right\},
\end{equation}
\item[(ii)] from \eqref{B-I-2'}, $Z_1$ is determined by $Z_2$ and it can be seen that also $Z_1\in L_\mu$.
\item[(iii)] In addition, when $\lambda = \frac{1}{2}$,
it follows from \eqref{B-I-1} that $Z_2$ is not only in $L_{-1}$ but also in $\ker(\overline{K}_{\overline{\mathcal{V}},\overline{\mathcal{Z}}}^*)$.
\end{itemize}

\begin{rem}
Since $\mu\in [-1, 0]$ and $\frac{F^3}{a^2|\mathcal{V}|^4|\mathcal{Z}|^2}>1$,
we find that $1\ge\lambda\ge \lambda_0 (>0)$, where $\lambda_0$ is the solution of the cubic equation
\begin{equation}
\left(\lambda-\dfrac{1}{2}\right)^2(\lambda-1)=-\dfrac{1}{8}.
\end{equation}
\end{rem}

The condition that $U\in\ker(\ad(\mathcal{V}))_0$ must satisfy depends on whether the value of $\lambda$ is $\frac{1}{2}$ or not.  
In the case $\lambda = \frac{1}{2}$,
substituting \eqref{B-I-2'} into \eqref{B-I-3}, and from \eqref{bseq} and $K_{\overline{\mathcal{V}},\overline{\mathcal{Z}}}^2 Z_2=-Z_2$,
we have $\overline{K}_{\overline{\mathcal{V}},\overline{\mathcal{Z}}}U=0$,
i.e., $U\in \ker(\ad(\overline{\mathcal{V}}))_0\cap \ker(\ad(J_{\overline{\mathcal{Z}}}\overline{\mathcal{V}}))_0$.
In the case of $\lambda\ne\dfrac{1}{2}$,
from \eqref{B-I-1}, we find that $U$ can be expressed using $Z_2$ as
\begin{equation}
U=-\dfrac{b_4}{\lambda-\frac{1}{2}}\overline{K}_{\overline{\mathcal{V}},\overline{\mathcal{Z}}}^*Z_2.
\end{equation}

\begin{rem}
From the definition of $b_1$, we have $0 < b_1 < 1$, so it is possible that $\lambda = b_1$.  
However, in this case, $Z_2 = 0$ follows from equation \eqref{B-I-2'}, $U = 0$ follows from \eqref{B-I-1},
and $Z_1 = 0$ follows from \eqref{B-I-3}.  
Therefore, $b_1$ cannot be an eigenvalue of $B$.
\end{rem}


\begin{rem}
Since $H_\theta$ is positive semi-definite,
it suffices to show that the determinant of $B$ is positive to show that all eigenvalues of $H_\theta$ are positive.
From properties of a determinant and spectral analysis of $K^2$,
we have
\begin{align*}
\det(B)
=&\det\left(\begin{array}{ccc}
\frac{1}{2}\mathrm{id}&
O&
-b_4 \overline{K}^\ast
\\
O&
b_1\mathrm{id}
&
\ad(\overline{\mathcal{V}})^\ast \circ \left(b_3\,\mathrm{id}-b_4 K\right)
\\
-b_4 \overline{K}&
\left(b_3\,\mathrm{id}+b_4K)\circ \ad(\overline{\mathcal{V}}\right)
&
b_2\,\mathrm{id}
\end{array}\right)\\
=&\det\left(\begin{array}{ccc}
\frac{1}{2}\mathrm{id}&
O&
-b_4 \overline{K}^*
\\
O&
b_1\mathrm{id}
&
\ad(\overline{\mathcal{V}})^\ast \circ \left(b_3\,\mathrm{id}-b_4 K\right)
\\
O&
\left(b_3\,\mathrm{id}+b_4K)\circ \ad(\overline{\mathcal{V}}\right)
&
b_2\,\mathrm{id} - 2b_4^2 \overline{K}\,\overline{K}^\ast
\end{array}\right)\\
=&\det\left(\begin{array}{ccc}
\frac{1}{2}\mathrm{id}&
O&
-b_4 \overline{K}^*
\\
O&
b_1\mathrm{id}
&
\ad(\overline{\mathcal{V}})^\ast \circ \left(b_3\,\mathrm{id}-b_4 K\right)
\\
O&
O
&
b_2\,\mathrm{id} - 2b_4^2 \overline{K}\,\overline{K}^\ast
-\frac{1}{b_1}\left(b_3\,\mathrm{id}+b_4K\right)\circ \left(b_3\,\mathrm{id}-b_4 K\right)
\end{array}\right)\\
=&\dfrac{1}{2^{k-m-1}}
\det\left(
b_1b_2\,\mathrm{id} - 2b_1b_4^2 \left(\mathrm{id}+K^2\right)
-\left(b_3^2\,\mathrm{id}-b_4^2K^2\right)
\right)\\
=&\dfrac{1}{2^{k-m-1}}
\det\left(
(b_1b_2-b_3^2-b_4^2) \mathrm{id} - 2b_1b_4^2 \left(\mathrm{id}+K^2\right)
+b_4^2\left(\mathrm{id}+K^2\right)
\right)\\
=&\dfrac{1}{2^{k-m-1}}
\det\left(
\dfrac{1}{2}\mathrm{id} - (2b_1-1)b_4^2 \left(\mathrm{id}+K^2\right)
\right)\\
=&\dfrac{1}{2^{k-2}}
\det\left(
\mathrm{id} - \dfrac{a^2 |\mathcal{V}|^4 |\mathcal{Z}|^2}{2F^3}\left(\mathrm{id}+K^2\right)
\right)
\end{align*}
Since the eigenvalues $\mu$ of $K^2$ satisfies $-1\le \mu\le 0$ and $\dfrac{a^2 |\mathcal{V}|^4 |\mathcal{Z}|^2}{2F^3}<\dfrac{1}{2}$ holds,
we find that $\det(B)>0$.
\end{rem}

\end{document}